\documentclass[11pt,letterpaper]{amsart}
\usepackage[matrix,arrow,tips,curve]{xy}

\usepackage{color}
 
\numberwithin{equation}{section}

 \usepackage{color}
 \usepackage{cases}
 \usepackage{amsmath}
 \usepackage[latin1]{inputenc}
\usepackage[T1]{fontenc}
\usepackage{verbatim}
\usepackage{graphicx}
\def\M{\mathcal{M}}

\def\C{\mathcal{C}}

\def\ZZ{\mathbb{Z}}

\def\Pic{\operatorname{Pic}}

\def\oc2{\mathcal{O}_{\C_2}}

\newtheorem{thm}{Theorem}[section]
\newtheorem*{thm*}{Theorem}

\newtheorem{lem}[thm]{Lemma}
\newtheorem{prop}[thm]{Proposition}
\newtheorem{conj}[thm]{Expectation}

\theoremstyle{definition}
\newtheorem{df}{Definition}

\newtheorem{question}{Question}

\newcommand{\be}{\begin{equation}}
\newcommand{\ee}{\end{equation}}

\begin{document}

\title {Double covers of curves on Nikulin surfaces}

\dedicatory{Dedicated to Peter Newstead on the occasion of his $80^{th}$ birthday}

\author{Simona D'Evangelista}
\address{Simona D'Evangelista: DISIM, Universit\`a degli Studi dell'Aquila
\hfill \newline\texttt{}  \indent Via Vetoio I, Loc. Coppito, 67100 L'Aquila, Italy} \hfill \newline\texttt{}
 \email{{\tt simona\_devangelista@hotmail.it}}

\author{Margherita Lelli-Chiesa}
\address{Margherita Lelli-Chiesa: Dipartimento di Matematica e Fisica, Universit\`a Roma Tre
\hfill \newline\texttt{}  \indent Largo San Leonardo Murialdo 1, 00146 Roma, Italy} \hfill \newline\texttt{}
 \email{{\tt margherita.lellichiesa@uniroma3.it}}
\begin{abstract}
We survey basic results concerning Prym varieties, the Prym-Brill-Noether theory initiated by Welters, and Brill-Noether theory of general étale double covers of curves of genus $g\geq 2$. We then specialize to curves on Nikulin surfaces and show that étale double covers of curves on Nikulin surfaces of standard type do not satisfy Welters' Theorem. On the other hand, by specialization to curves on Nikulin surfaces of non-standard type, we prove that general double covers of curves ramified at $b=2,4,6$ points are Brill-Noether general; the case $b=2$ was already obtained by Bud \cite{bud} with different techniques.
\end{abstract}

\maketitle
\section{Introduction}
Double covers of complex curves are a classical and still very hot topic in algebraic geometry. Part of the interest in them stems from the fact that any étale double cover $\pi:\widetilde C\to C$ of a smooth irreducible curve $C$ of genus $g\geq 2$ naturally defines a principally polarized abelian variety $(P,\Xi)$ of dimension $g-1$, known as the Prym variety of $\pi$. Prym varieties were introduced by  Schottky and Jung \cite{SJ} in relation to the Schottky problem and were named after the German mathematician Prym by Mumford (\cite{Mumford}), who was the first to investigate them from an algebraic point of view. Despite a vast literature on the topic, many questions concerning the geometry of general double covers remain open, both in the étale and in the ramified case. We will here focus on Brill-Noether type questions. 

We recall that the Brill-Noether theory of a general curve $C$ of genus $g$ is quite well understood since the 1980s. The Brill-Noether Thorem establishes that the Brill-Noether variety $W^r_d(C)$, parametrizing degree $d$ line bundles on $C$ with a space of global sections of dimension $\geq r+1$, is nonempty if and only if the so-called Brill-Noether number $\rho(g,r,d):=g-(r+1)(g-d+r)$ is nonnegative. Furthermore, if nonempty, $W^r_d(C)$ is  smooth of dimension $\rho(g,r,d)$ outside of $W^{r+1}_d(C)$ by the Gieseker-Petri Theorem. The existence part of the Brill-Noether Theorem is actually valid for any smooth curve of genus $g$ and is due to Kleiman-Laksov \cite{KL} and Kempf \cite{Ke}. The first proofs of both the non-existence part of the Brill-Noether Theorem and of the Gieseker-Petri Theorem, all based on degeneration techniques, were achieved by Griffiths-Harris \cite{griffiths}, Eisenbud-Harris \cite{EH,Eisenbud}, and Gieseker \cite{Gi}. A big breakthrough came with Lazarsfeld's alternative proof of the Gieseker-Petri Theorem \cite{Lazarsfeld,Pareschi}, that avoided any type of degeneration and proceeded instead by specialization to smooth curves lying on $K3$ surfaces.

If $\pi:\widetilde C\to C$ is a double cover of a general genus $g\geq 2$ curve $C$, the curve $\widetilde C$ is obviously non-general  in moduli and it thus makes sense to investigate its Brill-Noether theory. In the étale case $\widetilde C$ has genus $2g-1$ and is never Petri general because $W^1_{2g-2}(\widetilde C)$ turns out to be singular (cf. \cite{Welters} and Section \ref{inv} below). Moreover, if $g$ is even, $\widetilde C$ possesses a pencil of degree $g$ which prevents it from being Brill-Noether general (cf. \cite{generalcovers} and Section \ref{tre} later in this paper). Hence, the Brill-Noether behaviour of $\widetilde C$ seems quite wild. The picture is much nicer if one considers instead the so-called Prym-Brill-Noether theory of $\widetilde C$. This is the study of the geometry of the Prym-Brill-Noether varieties $V^r$, that were introduced by Welters \cite{Welters} and are related to the geometry of the Prym variety $P$; indeed, they are obtained by intersecting the Brill-Noether varieties $W^r_{2g-2}(\widetilde C)$ with an appropriate translate of $P$ living in $\Pic^{2g-2}(\widetilde C)$. Using the degeneration techniques developed by Eisenbud and Harris, Welters proved (cf. Thm. \ref{Prym} below) that for a general $\pi$ the varieties $V^r$ are smooth of dimension $\rho^-(g,r):=g-1- (r+1)r/2$. In particular, they are empty if $\rho^-(g,r)<0$; conversely, $V^r$ is nonempty if $\rho^-(g,r)\geq 0$ for any étale double cover $\pi$ by a result of Bertram \cite{Bertram}. These statements are the analogues in Prym-Brill-Noether theory of the Brill-Noether Theorem and the Gieseker-Petri Theorem.

Perhaps discouraged by the intricate picture in the étale case, people ignored the Brill-Noether theory of $\widetilde C$ in the ramified case until recently, where Bud \cite{bud} proved that $\widetilde C$ is both Brill-Noether and Petri general (that is, it satisfies both the the Brill-Noether Theorem and the Gieseker-Petri Theorem) if $\pi:\widetilde C\to C$ is a double cover of a general genus $g$ curve branched over $2$ general points of $C$. This result is quite amazing as it highlights that ramified double covers behave better then étale ones from a Brill-Noether viewpoint. 

This paper was born from the desire to make $K3$ surfaces and Lazarsfeld's techniques enter the picture. This is ensured by specialization to Nikulin surfaces, that is, primitively polarized $K3$ surfaces $(S,H)$ endowed with a double cover $\pi_S:\widehat S\to S$ ramified along eight ($-2$)-curves, which are both pairwise disjoint and disjoint from smooth curves in the linear system $|H|$. The minimal model of $\widehat S$ is again a $K3$ surface $\widetilde S$ endowed with an involution having $8$ fixed points. The Picard ranks of both $S$ and $\widetilde S$ are $\geq 9$ and, when equality holds, the Picard groups have been described in \cite{VGS, GS}. It turns out that there are two types of Nikulin surfaces, depending on whether the embedding of the rank $9$ sublattice generated by $H$ and the ($-2$)-curves in $\Pic(S)$ is primitive or not: according to the terminology introduced in \cite{KLV1},  $(S,H)$ is called of standard type in the former case, and of non-standard type in the latter. Note that $\pi$ induces an étale double cover $\pi|_C:\widetilde{C}:=\pi^{-1}(C)\to C$ of any smooth curve $C\in|H|$, and the curve $\widetilde C$ can be identified with its image in $\widetilde S$. In the last decade Nikulin surfaces have been largely exploited in the study of étale double covers of curves (or equivalently, Prym curves) and their moduli space $\mathcal R_h$. In particular, Farkas and Verra \cite{FV} proved that for $6 \neq h\leq 7$, general étale double covers of genus $h$ curves live on Nikulin surfaces of standard type, and used this fact to describe the birational geometry of $\mathcal R_h$ in low genera. Furthermore, in the standard case the curve $\widetilde C\subset S$ has the gonality of a general étale double cover of a genus $h$ curve \cite{generalcovers}. It is thus natural to wonder whether a different proof of Welters' Theorem can be obtained by specialization to étale double covers  living on Nikulin surfaces of standard type. Unfortunately, we will answer negatively to this question proving the following theorem.
\begin{thm*}
Let $\pi:\widetilde C\to C$ be an étale double cover on a  Nikulin surface $S$ of standard type, with $C$ in the primitive linear system $|H|$ of genus $h$. If $h>7$ or $h=6$, the curve $\widetilde C$ does not satisfy Welters' Theorem.
\end{thm*}

Nikulin surfaces of non-standard type have not received as much attention, and their investigation essentially started in \cite{KLV1,KLV2}. As emerged in \cite{KLV1}, in the non-standard case the curve $\widetilde C$ is quite special as it possesses two theta characteristics with many sections cut out by two line bundles $R_1,R_2\in \Pic(\widetilde S)$. The main result of the paper suggests that Nikulin surfaces of non-standard type are instead the right environment to investigate double covers of curves ramified at $2,4,6$ points, which can be realized on a Nikulin surface of non-standard type by restricting $\pi_S$ to the inverse image of smooth curves in the linear systems $|R_1|$ and $|R_2|$. By specialization to ramified double covers of curves on Nikulin surfaces of non-standard type, we obtain the following generalization of Bud's result.
\begin{thm*}
Let $\pi:\widetilde C\to C$ be a general double cover of a genus $g\geq 2$ curve ramified at $2$, $4$, or $6$ points. Then the curve $\widetilde C$ is Brill-Noether general.
\end{thm*}
The paper is organized as follows. Section \ref{uno} surveys the basic theory of Prym varieties and Prym-Brill-Noether theory, with particular attention to Welters' infinitesimal study of $V^r$ in terms of the Prym-Petri map, which is the anti-invariant part $\mu_{0,L}^-$ of the Petri map $\mu_{0,L}$ of any line bundle $L\in V^r$. The sections ends with a discussion and interpretation of its invariant counterpart $\mu_{0,L}^+$, as well as with a conjectural picture concerning its injectivity. Section \ref{due} recalls some prerequisites on theta characteristics in order to show, following an argument by Beauville \cite{Beauville}, that for any irreducible étale double cover $\pi:\widetilde C\to C$ the curve $\widetilde C$ possesses some invariant vanishing thetanulls: in particular, $\widetilde C$ is not Petri general. In Section \ref{tre} we concentrate on the Brill-Noether theory of $\widetilde C$ when $\pi$ is general. More precisely, we recall Schwarz's non-existence result \cite{Schwarz} concerning linear series whose Brill-Noether number is negative enough, and Aprodu-Farkas' theorem \cite{generalcovers} on the gonality of $\widetilde C$. The latter prevents $\widetilde C$ from being Brill-Noether general if the genus of $C$ is even, while the Brill-Noether generality/speciality in the odd genus case is still unknown in full generality, up to our knowledge. We also recall Bud's result concerning the ramified case. Section \ref{nik} is focused on Nikulin surfaces of standard and non-standard type and on the proof of the above theorems.

\section{Prym-Brill-Noether theory}\label{uno}
\subsection{Preliminaries on Prym varieties}
Prym varieties are principally polarized abelian varieties arising from \'{e}tale double covers of curves, and they are useful to link the geometry of curves to that of abelian varieties.

Let $C$ be a complex smooth irreducible curve, and let $\pi:\widetilde{C}\rightarrow C$ be an irreducible \'{e}tale double cover of $C$. Denoting by $g$ and $\widetilde{g}$ the genus of $C$ and $\widetilde{C}$, respectively, Hurwitz's formula yields $\widetilde{g}=2g-1$. The cover $\pi$ induces the so-called {\em Norm map} between the Jacobians $J(C)$ and $J(\widetilde{C})$ of $C$ and $\widetilde{C}$, respectively:

\begin{align*}
\mathrm{Nm} \colon J(\widetilde{C})& \longrightarrow J(C) \\
\mathcal O_{\widetilde C}(D) &\longmapsto  \mathcal O_{C}(\pi(D)).
\end{align*}

This fits in the following commutative diagram:
 $$\xymatrix{ \widetilde{C} \ar[d]_{\pi} \ar[r]^{a_{y_0}}
    & J(\widetilde{C})\ar[d]^{\mathrm{Nm}}\\
     C \ar[r]_{a_{x_0}} & J(C),}$$
where  $a_{x_0}$ and $a_{y_0}$ are the Abel-Jacobi maps associated with some fixed points $x_0 \in C$ and $y_0 \in \widetilde{C}$ such that $\pi(y_0)=x_0$.

In particular, the principally polarized Abelian varieties $(J(C),\Theta)$ and $(J(\widetilde{C}),\widetilde{\Theta})$ are related by two maps
$$ \pi^*:J(C)\rightarrow J(\widetilde{C}), \quad \mathrm{Nm}:J(\widetilde{C})\rightarrow J(C),$$
such that $\pi^*$ and $\mathrm{Nm}$ are dual to each other and
$\mathrm{Nm}\circ \pi^*:J(C)\rightarrow J(C) $ is multiplication by two (\cite{Mumford}). Indeed, if $N=\mathcal O_C(D)\in J(C)$, then $\pi^*N =\mathcal O_{\widetilde C}(\pi^{-1}(D))\in J(\widetilde C)$ and thus $\mathrm{Nm}(\pi^*N)=\mathcal O_C(\pi(\pi^{-1}(D)))\in J(C)$; since $\pi$ is a double cover, we get $\pi(\pi^{-1}(D))=2D$. \\
Denote by $\iota:\widetilde{C}\rightarrow \widetilde{C}$ the involution that interchanges the sheets of $\pi$. For every divisor $\widetilde D$ on $\widetilde{C}$, the following equality is straightforward:
$$\pi^{-1}(\pi(\widetilde D))=\widetilde D+\iota(\widetilde D).$$
It follows that 
$$\pi^*(\mathrm{Nm}(M))=M\otimes\iota^*(M), \quad \forall \quad M\in J(\widetilde{C}).$$ 
Since $\mathrm{Nm}$ is surjective, we also get that
$$\iota^*(\pi^*(N))=\pi^*N, \quad \forall \quad N \in J(C),$$
that is, the involution $\iota^*$ acts as the identity on $\pi^*(J(C))\subset J(\widetilde C)$. On the other hand, $\iota^*$ acts as $-1$ on $\mathrm{Ker}(\mathrm{Nm})\subset J(\widetilde C)$. 
In \cite{Vr1} Mumford decomposed the kernel of $\mathrm{Nm}$ into two irreducible components:
$$P_0:=\{M\otimes \iota^*M^\vee\,\,|\,\, M\in \Pic^0(\widetilde C)\},$$
$$P_1:=\{M\otimes \iota^*M^\vee\,\,|\,\, M\in \Pic^1(\widetilde C)\}.$$
More precisely, he proved the following:
\begin{lem}[\cite{Vr1} Lem. 1]
If $L$ is a line bundle on $\widetilde C$ such that $\mathrm{Nm}\,L\simeq \mathcal{O}_C$, then $L\simeq M \otimes \iota^*M^\vee$ for some line bundle $M$ on $\widetilde C$. Moreover, $M$ can be chosen of degree $0$ or $1$.
\end{lem}
A classical theorem by Wirtinger (\cite{Wirtinger}) yields that the dimension of the space of global sections is constant mod 2 on $P_0$ and $P_1$ and it has opposite parity on the two components.
The component $P:=P_0$ of $\mathrm{Ker}(\mathrm{Nm})$ containing the origin is an abelian subvariety of $J(\widetilde C)$ of dimension $$\mathrm{dim}\,P=\mathrm{dim}\,J(\widetilde{C})-\mathrm{dim}\,J(C)=g-1;$$ furthermore, it turns out that the canonical pola\-ri\-zation of $J(\widetilde{C})$ restricts to twice a principal polarization $\Xi$ on $P$. The principally polarized abelian variety $(P,\Xi)$ is called the {\em Prym variety} associated with $\pi$ (cf. \cite[\S 3]{Mumford} for more details).  
        
By the same argument \cite{Mumford}, the inverse image $\mathrm{Nm}^{-1}(\omega_C)\subset \mathrm{Pic}^{2g-2}(\widetilde{C})$ of the canonical line bundle $\omega_C \in \mathrm{Pic}^{2g-2}(C)$ breaks up in two components
$$P^{+}:=\{L\in \mathrm{Pic}^{2g-2}(\widetilde {C}) \ | \ \mathrm{Nm}(L)\simeq \omega_C,\,\, h^0(C,L)\equiv 0\,\, \mbox{(mod}\ 2)\},$$
$$P^{-}:=\{L\in \mathrm{Pic}^{2g-2}(\widetilde {C}) \ | \ \mathrm{Nm}(L)\simeq \omega_C,\,\, h^0(C,L)\equiv 1\,\,\mbox{(mod} \ 2)\},$$
 which are both translates of $P$.
 
 We recall that the cover $\pi:\widetilde C \rightarrow C$ defines a class of order two $\eta \in J(C)[2]$ such that $\pi_*\mathcal O_{\widetilde C}=\mathcal O_C\oplus\eta$; viceversa, any non-trivial $\eta \in J(C)[2]$ determines an étale double cover of $C$ by setting $\widetilde C:=\mathrm{Spec}(\mathcal O_C\oplus\eta)$. Therefore, the datum of $\pi$ is equivalent to that of the pair $(C,\eta)$, which is called a {\em Prym curve} of genus $g$. The moduli space 
 $$
 \mathcal R_g:=\left\{(C,\eta)\,|\, C\textrm{ smooth curve of genus }g,\,\,\eta\in J(C)[2]\,,\,\eta\not\simeq\mathcal O_C    \right\},
 $$
 has attracted much attention in the last decades; we refer to \cite{gabi} for a very nice survey on its geometry.

\subsection{Prym-Brill-Noether varieties and Welters' Theorem} 

In \cite{Welters} Welters initiated a Brill-Noether study of the varieties $P^-$ and $P^+$ and introduced the following closed subsets of $\mathrm{Nm}^{-1}(\omega_C)$, defined for any integer $ r\geq -1 $:
\begin{multline*}V^r(C,\eta):=\{L \in \mathrm{Pic}^{2g-2}(\widetilde C) \ | \  \mathrm{Nm}(L)\simeq \omega_C,\, h^0(C,L)\geq r+1,\, h^0(C,L)\equiv r+1\,\, \mbox{(mod}\ 2)\}.\end{multline*}
Scheme-theoretically, these can be realized as the intersections
\begin{align}\label{inter}V^r(C,\eta)=W^r_{2g-2}(\widetilde C)\cap P^{+} \quad \mbox{if} \ r \ \mbox{is odd},\\\nonumber
V^r(C,\eta)=W^r_{2g-2}(\widetilde C)\cap P^{-} \quad \mbox{if} \ r \ \mbox{is even},
\end{align}
where $$W^r_{2g-2}(\widetilde C)=\left\{ L \in \mathrm{Pic}^{2g-2}(\widetilde C) \ | \  h^0(C,L)\geq r+1  \right\}$$ is the classical Brill-Noether variety; this justifies the name  {\em Prym-Brill-Noether varieties} used for the loci $V^r(C,\eta)$. The following result proved in \cite{Vr1} and \cite{Vr2} concerns their expected dimension.
\begin{prop}
Let $L \in \mathrm{Nm}^{-1}(\omega_C)$ and set $h^0(L)=r+1$. Then, the dimension of $V^r(C,\eta)$ at $L$ satisfies
$$\mathrm{dim}_L (V^r(C,\eta)) \geq g-1- \binom{r+1}{2}.$$
\end{prop} 
We will refer to the number
\begin{equation*}\label{PBN}
\rho^-(g,r):=g-1- \binom{r+1}{2},
\end{equation*}
as the {\em Prym-Brill-Noether number}.

A priori one may hope that classical results of Brill-Noether theory carry over to \'{e}tale double covers of smooth curves and to the varieties $V^r(C,\eta)$, but unfortunately this is not the case. Indeed, many statements of this theory, such as the Brill-Noether Theorem and the Gieseker-Petri Theorem, hold only for a general curve $\widetilde C$ of genus $\widetilde g$ and they fail in the case where $\widetilde C$ is the \'{e}tale double cover of a genus $g$ curve. Nevertheless, an analogue of the Gieseker-Petri Theorem for Prym varieties was obtained by Welters \cite{Welters} by a degeneration argument similar to the one used by Eisenbud and Harris \cite{Eisenbud} in their proof of the classical statement.

 Since $\omega_{\widetilde C}=\pi^* \omega_C$, the push-pull formula yields the decomposition
 \begin{equation}\label{decomp}H^0(\widetilde C,\omega_{\widetilde C})=H^0(C,\omega_C)\oplus H^0(C,\omega_C(\eta)).\end{equation}
into invariant and anti-invariant forms under the action of $\iota$. Fix a line bundle $L \in \mathrm{Nm}^{-1}(\omega_C)$ and set $h^0(L)=r+1$. Since the composition $\mathrm{Nm}\circ \pi^*$ is the multiplication by $2$, the differential of $\mathrm{Nm}$ at $L$ is $2(_{}^{t}\pi^*)$, where 
$$ _{}^{t}{\pi}^{*}:(H^0(\widetilde C,\omega_{\widetilde C}))^\vee\rightarrow (H^0(C,\omega_C))^\vee$$
is the transpose of the pull-back of differential forms. By taking kernels, we obtain the identification 
\begin{equation}\label{milan}T_L(\mathrm{Nm}^{-1}(\omega_C))=(H^0(C,\omega_C(\eta)))^\vee\hookrightarrow (H^0 (\widetilde C,\omega_{\widetilde C}))^\vee,\end{equation} where the inclusion is the transpose of the projection map
\[
\begin{aligned}
p\colon H^0(\widetilde C,\omega_{\widetilde C}) &\longrightarrow H^0(C,\omega_C(\eta))\\
\lambda &\longmapsto \frac{1}{2}(\lambda-\iota^* \lambda)
\end{aligned}.
\]
By classical Brill-Noether theory (cf., e.g., \cite{Arbarello}), we know that 
\begin{equation}\label{tw}T_L(W^r_{2g-2}(\widetilde C))=(\mathrm{Im} \mu_{0,L})^\bot \subset (H^0 (\widetilde C,\omega_{\widetilde C}))^\vee,\end{equation}
where the so-called {\em Petri map} $$\mu_{0,L}:H^0(\widetilde{C}, L) \otimes H^0(\widetilde{C}, \omega_{\widetilde{C}} \otimes L^\vee)\rightarrow H^0(\widetilde{C},\omega_{\widetilde{C}})$$ is multiplication of global sections. In particular, the dimension of the tangent space at $L$ to $W^r_{2g-2}(\widetilde C)$ is:
\begin{equation}\label{dimw}
\mathrm{dim}\,T_L(W^r_{2g-2}(\widetilde C))= \rho(2g-1,r,2g-2) + \mathrm{dim(Ker}\mu_{0,L}),
\end{equation}
where the number $\rho(2g-1,r,2g-2):=2g-1-(r+1)^2$ equals the difference between the dimensions of the codomain and the domain of $\mu_{0,L}$.
Recalling \eqref{inter} and \eqref{milan}, we obtain:
\begin{equation}\label{tg}T_L(V^r(C,\eta))=(\mathrm{Im} \mu_{0,L})^\bot \cap (H^0(C,\omega_C(\eta))^\vee=(\mathrm{Im}(p\circ \mu_{0,L}))^\bot.\end{equation} 
By applying $\pi^*$ to the isomorphism $\mathrm{Nm}\,L\simeq \omega_C$, one gets $L\otimes \iota^* L\simeq \omega_{\widetilde C}$, or equivalently, $\omega_{\widetilde C}\otimes L^\vee\simeq \iota^* L$. Consider the following composition of maps:

   \begin{equation*}\label{composition}\xymatrix{ H^0(\widetilde C,L)\otimes H^0(\widetilde C,L) \ar[r]^(.45){1\otimes \iota^*}
    & H^0(\widetilde C,L)\otimes H^0(\widetilde C,\omega_{\widetilde C}\otimes L^\vee) \ar[r]^(.65){\mu_{0,L}} & H^0(C,\omega_{\widetilde C}) \ar[r]^{p} & H^0( C,\omega_C(\eta))\\
s\otimes t \ar@{|->}[r] & s\otimes \iota^* t \ar@{|->}[r] & s \cdot \iota^* t \ar@{|->}[r] & \frac{1}{2}(s\cdot (\iota^* t)-(\iota^* s) \cdot t),}\end{equation*}
which is clearly skew-symmetric. By restriction to $\large \wedge^2 H^0(\widetilde C,L)$, we thus obtain the map
\[
\begin{aligned}
\mu_{0,L}^- \colon \large \wedge^2 H^0(\widetilde C,L) &\to H^0( C,\omega_{ C}(\eta)) \\
s\wedge t &\mapsto \frac{1}{2}(s \cdot (\iota^* t)-(\iota^* s) \cdot t),
\end{aligned}
\]
which is called the {\em Prym-Petri map} of $L$. We can rewrite \eqref{tg} in terms of $\mu_{0,L}^-$ as $$T_L(V^r(C,\eta))=(\mathrm{Im} \mu_{0,L}^-)^\perp.$$ Since $\rho^-(g,r)$ equals the difference between the dimensions of the codomain and the domain of $\mu_{0,L}^-$, we conclude that 
\begin{equation}\label{dimv}
\mathrm{dim}(T_L(V^r(C,\eta)))= \rho^-(g,r) + \mathrm{dim(Ker}\mu_{0,L}^-),
\end{equation}
and $V^r(C,\eta)$ is smooth of dimension $\rho^-(g,r)$ at $L$ if and only if the Prym-Petri map is injective. Welters proved that this is the case for all $L \in \mathrm{Nm}^{-1}(\omega_C)$ provided that the cover $\pi$ is general. 
\begin{thm}[\cite{Welters} Thm. 1.11]
\label{Prym}
Let $(C,\eta)\in \mathcal R_g$ be general and let $\pi: \widetilde C\rightarrow C$ be the \'{e}tale double cover defined by $\eta$. Then the Prym-Petri map $\mu_{0,L}^-$ is injective for all $L \in \mathrm{Nm}^{-1}(\omega_C)$.
\end{thm}
The result is analogous to the Gieseker-Petri Theorem, as it yields the smoothness of the Prym-Brill-Noether varieties $V^r(C,\eta)$ for a general $(C,\eta)$ and their emptiness in the cases where $\rho^-(g,r)<0$. The analogue of the existence part of the Brill-Noether Theorem for any Prym curve was instead established by Bertram:
\begin{thm}[\cite{Bertram} Thm. 1.4]\label{be}
Let $(C,\eta)$ be a Prym curve of genus $g$ and let $\pi: \widetilde C\rightarrow C$ be the \'{e}tale double cover defined by $\eta$. If $\rho^-(g,r)\geq0$, the Prym-Brill-Noether variety $V^r(C,\eta)$ is nonempty.
\end{thm}

Welters' proof is by degeneration to a reduced nodal curve $C_0$ consisting of a string of rational curves and $g$ elliptic curves $E_1,...,E_g$, such that, if $x_i$ and $y_i$ are the two intersection points of any $E_i$ with two adjacent components, the line bundle $\mathcal O_{E_i}(x_i-y_i)$ is not a torsion point in $\mathrm{Pic}^0(E_i)$. Consider a line bundle $\eta_0$ on $C_0$ which is trivial on every component except on $E_g$, where it restricts to a non-zero point of $\mathrm{Pic}^0(E_g)[2]$. The étale double cover $\widetilde C_0$ of $C_0$ defined by $\eta_0$ then looks as follows:
\begin{center}
\includegraphics[width=%
1.0\textwidth]{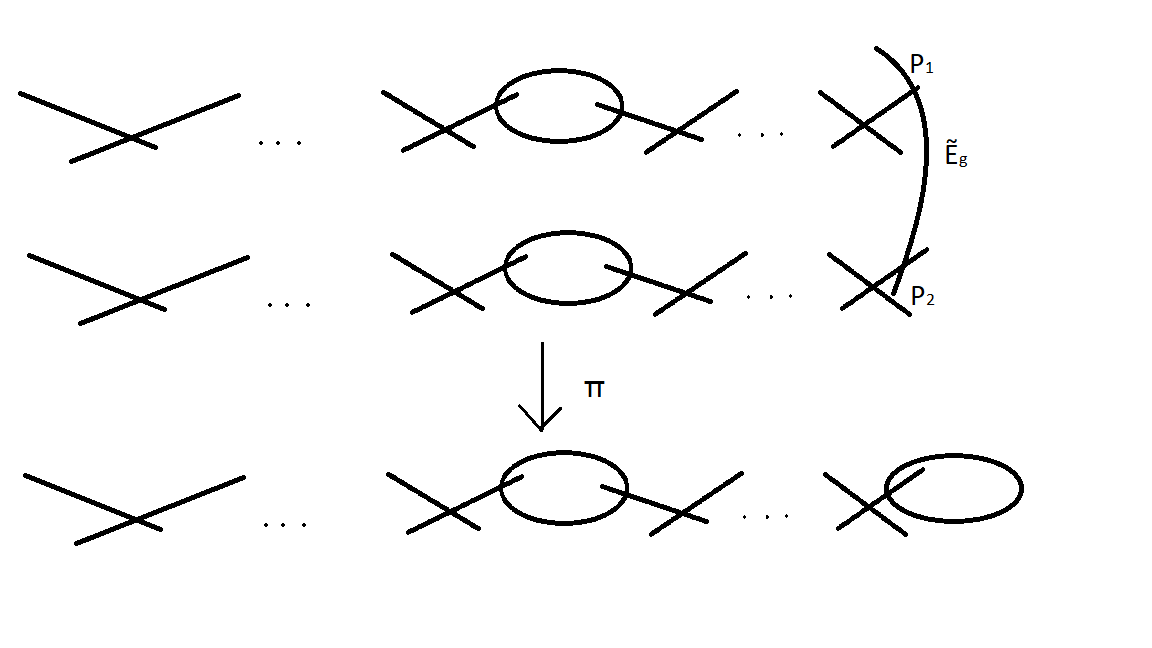}
\end{center}
The points $P_1$ and $P_2$ on $\widetilde{E}_g$ are not $\mathbb{Z}$-independent; indeed, $\mathcal O_{\widetilde E_g}(P_1-P_2)\in \mathrm{Pic}^0(\widetilde E_g)$ is a $2$-torsion point. 
This is why Eisenbud and Harris' proof \cite{Eisenbud} of the Gieseker-Petri Theorem fails for $\widetilde C_0$. However, using the theory of limit linear series and the Brill-Noether theory of a (non general) $2$-pointed elliptic curve, Welters wrote down the explicit form of any element $\rho$ in the kernel of the relevant Petri maps and, thanks to the skew-symmetry of $1\otimes \iota^*$, concluded that such a $\rho$ never lies in the kernel of the corresponding Prym-Petri map.

\subsection{The map $\mu_{0,L}^+$}\label{più}
The decomposition \eqref{decomp} provides a splitting of the Petri map $\mu_{0,L}$ of any line bundle $L \in \mathrm{Nm}^{-1}(\omega_C)$ into a $\iota$-anti-invariant part, namely, the Prym-Petri map $\mu_{0,L}^-$, and a $\iota$-invariant part
\[
\begin{aligned}
\mu_{0,L}^+\colon \mathrm{Sym}^2 H^0(\widetilde{C},L) &\to H^0(C,\omega_{ C}) \\
s\otimes t+t \otimes s &\mapsto s \cdot (\iota^* t)+(\iota^* s) \cdot t.
\end{aligned}
\]
As already remarked, the Prym-Petri map $\mu_{0,L}^-$ governs the smoothness of the varieties $V^r(C,\eta)$ and has thus been extensively investigated. By contrast, the map $\mu_{0,L}^+$ has been almost ignored so far. Up to our knowledge, it only appeared in the proof of the uniruledness of $\mathcal R_8$ due to Farkas and Verra \cite{FV}.

We provide an interpretation of the map $\mu_{0,L}^+$ by considering the inclusion $V^r(C,\eta)\subset W^r_{2g-2}(\widetilde C)$ and the induced short exact sequence of linear maps
$$0\rightarrow T_L(V^r(C,\eta))\rightarrow T_LW^r_{2g-2}(\widetilde C)\rightarrow N_{V^r(C,\eta)/W^r_{2g-2}(\widetilde C),L}\rightarrow 0,$$
where $N_{V^r(C,\eta)/W^r_{2g-2}(\widetilde C),L}$ is the normal space at the point $L$ of $V^r(C,\eta)$ in $W^r_{2g-2}(\widetilde C)$. By \eqref{dimw} and \eqref{dimv}, we get:
\begin{equation*}
\begin{split}
\mathrm{dim}_{\mathbb C}&(N_{V^r(C,\eta)/W^r_{2g-2}(\widetilde C),L})=\\&=\rho(2g-1,r,2g-2)+\mathrm{dim(Ker}\mu_{0,L})- \rho^-(g,r) - \mathrm{dim(Ker}\mu_{0,L}^-)=\\
&=g-\frac{(r+1)(r+2)}{2}+\mathrm{dim(Ker}\mu_{0,L}^+),
\end{split}
\end{equation*}
where we have used that $\mathrm{Ker}\mu_{0,L}=\mathrm{Ker}\mu_{0,L}^+\oplus \mathrm{Ker}\mu_{0,L}^-$ by construction. We set $$\rho^+(g,r):=\rho(2g-1,r,2g-2)- \rho^-(g,r)=g-\frac{(r+1)(r+2)}{2}$$ and notice that this number equals the difference between the dimensions of the codomain and the domain of $\mu_{0,L}^+$. We conclude that the dimension of the normal space $N_{V^r(\widetilde X)/W^r_{2g-2}(\widetilde X),L}$ equals the dimension of the cokernel of $\mu_{0,L}^+$. More precisely, \eqref{tw} and \eqref{tg} imply that
$$
N_{V^r(\widetilde X)/W^r_{2g-2}(\widetilde X),L} =(\mathrm{Im} \mu_{0,L})^\bot \cap (H^0(\omega_C))^\vee=(\mathrm{Im}(q\circ\mu_{0,L}))^\bot,
$$
where $q: H^0(\widetilde C,\omega_{\widetilde C})\to H^0(C,\omega_C)$ is the projection mapping a form $\lambda$ to $\frac{1}{2}(\lambda+\iota^* \lambda)$. The composition
\begin{align*}q\circ \mu_{0,L}\circ(1\otimes \iota^*):H^0(\widetilde C,L)\otimes H^0(\widetilde C,L)&\longrightarrow H^0(C,\omega_C)\\
\end{align*}
maps a decomposable tensor $s\otimes t$  to the invariant form $s(\iota^*t)+(\iota^*s)t$ and is thus symmetric. Its restriction to $\mathrm{Sym}^2 H^0(\widetilde{C},L)$ coincides with the map $\mu_{0,L}^+$ and this provides the identification
$$
N_{V^r(\widetilde X)/W^r_{2g-2}(\widetilde X),L}=(\mathrm{Im} \mu_{0,L}^+)^\bot.$$

When $(C,\eta)\in \mathcal R_g$ is general, it is natural to wonder whether the map $\mu_{0,L}^+$ is injective for all $L\in \mathrm{Nm}^{-1}(\omega_C)$. Unfortunately, this trivially fails for any line bundle $L\in V^r(C,\eta)$ where $r$ is any fixed integer $r$ such that $\rho^-(g,r)\geq 0$ but $\rho(2g-1,r,2g-2)<0$; indeed, the Prym-Petri map $\mu_{0,L}^-$ is injective by Theorem \ref{Prym}, but our assumptions prevent $\mu_{0,L}^+$from being injective because $\dim\mathrm{Ker} \mu_{0,L}^+\geq -\rho^+(g,r)>0$. More generally, the map $\mu_{0,L}^+$ fails to be injective for any line bundle $L\in V^r(C,\eta)$ as soon as 
\begin{equation}\label{brutta}
\rho^-(g,r)>\max\{ -1,\rho(2g-1,r,2g-2)\},
\end{equation} 
which implies $\rho^+(g,r)<0$. The inequalities \eqref{brutta} can be rewritten only in terms of $\rho$ as 
\begin{equation}\label{ro}-r\leq \rho(2g-1,r,2g-2)<r.\end{equation} In this range, the equality $\dim\mathrm{Ker} \mu_{0,L}^+= -\rho^+(g,r)$ is the best one may hope for; since $V^r(C,\eta)$ is smooth at $L$ by Theorem \ref{Prym}, this hope is equivalent to the expectation that $W^r_{2g-2}(\widetilde C)$ and $V^r(C,\eta)$ coincide in a neighborhood of $L$. 
\begin{conj}
Let $(C,\eta)\in \mathcal R_g$ be general and let $\pi: \widetilde C\rightarrow C$ be the \'{e}tale double cover defined by $\eta$. Then, the following hold: 
\begin{itemize}
\item[(i)] If $-r\leq \rho(2g-1,r,2g-2) < r$, then $W^r_{2g-2}(\widetilde C)=V^r(C,\eta)$. In particular, for all $L\in V^r(C,\eta)$ one has $$\mathrm{dim(Ker}\mu_{0,L})=\mathrm{dim(Ker}\mu_{0,L}^+)=-\rho^+(g,r).$$
\item[(ii)] If $\rho(2g-1,r,2g-2)\geq r$, then both $\mu_{0,L}$ and $\mu_{0,L}^+$ are injective for all $L\in V^r(C,\eta)$.
\end{itemize}
\end{conj}
The situation in the remaining cases $ \rho(2g-1,r,2g-2)<-r$ is already clear, as $W^r_{2g-2}(\widetilde C)$ turns out to be empty by a more general result of Schwarz \cite{Schwarz} that we will recall in Section \ref{tre}.

\section{Theta-characteristics and vanishing thetanulls}\label{due}
\subsection{Preliminaries on theta-characteristics}
A \emph{theta-characteristic} on a smooth irreducible curve $C$ is a line bundle $\theta \in \mathrm{Pic}^{g-1}(C)$ such that $\theta^{\otimes 2} = \omega_C$. Theta-characteristics are of two types, called odd and even according to the parity of the dimension of the space of their global sections. The set of theta-characteristics on $C$, denoted by $\mathrm{Th}(C)$, is a principal homogeneous space for the space $J(C)[2]$ of two-torsion points in the Jacobian of $C$, that is, $J(C)[2]$ acts freely and transitively on $\mathrm{Th}(C)$. In particular, we have $|\mathrm{Th}(C)|=|J(C)[2]|=2^{2g}$.

The $\mathbb{F}_2$-vector space $J(C)[2]$ is endowed with a nondegenerate symplectic form $$\langle\cdot,\cdot\rangle:J(C)[2] \times J(C)[2]\rightarrow \mathbb{F}_2,$$
which is called \emph{Weil pairing} and is defined as follows (\cite{Vr1}). For any pair of points $\eta,\epsilon \in J(C)[2]$, one can write $\eta=\mathcal{O}_C(D)$ and $\epsilon=\mathcal{O}_C(E)$ for two divisors $D$ and $E$ on $C$ with disjoint support, and choose rational functions $f$ and $g$ on $C$ such that $\mathrm{div}(f)=2D$ and $\mathrm{div}(g)=2E$. 
The condition on the supports of $D$ and $E$ ensures the validity of the so-called \emph{Weil Reciprocity Law} (see \cite{Vr2}):
$$f((g))=g((f)),$$
where the evaluation of a rational function $h$ at a divisor $Z$ whose support is disjoint from the set of zeros and poles of $h$ is defined as $h(Z):=\prod_{p \in C} h(p)^{\mathrm{mult}_p Z}$.
Hence, we get
$$\frac{f(2E)}{g(2D)}={\left(\frac{f(E)}{g(D)}\right)}^2=1,$$
and $\frac{f(E)}{g(D)}=\pm 1$.
The value of the Weil pairing at the pair $(\eta,\epsilon)$ is given by the following formula:
$$(-1)^{\langle \epsilon, \eta\rangle}=\frac{f(E)}{g(D)},$$
and one can check that this definition is independent of both the divisors $D$, $E$ and the rational functions $f$, $g$.

Given a symplectic vector space $(V,\langle \cdot,\cdot \rangle)$ over $\mathbb{F}_2$, we denote by $Q(V)$ the set of quadratic forms on $V$ with fixed polarity equal to the symplectic form $\langle \cdot,\cdot \rangle$, that is, all functions $q:V\rightarrow \mathbb{F}_2$ that satisfy the identity
$$q(x+y)=q(x)+q(y)+\langle x,y \rangle, \quad \forall \ x,y \in V.$$
Having fixed a symplectic basis $(e_1,...,e_g,f_1,...,f_g)$ of $V$, the so-called \emph{Arf invariant} of a quadratic form $q \in Q(V)$ is defined as
$$\mathrm{arf}(q):=\sum_{i=1}^g q(e_i)\cdot q(f_i) \in \mathbb{F}_2,$$
this definition being independent of the basis. A quadratic form $q \in Q(V)$ is called {\em even} if $\mathrm{arf}(q)=0$, and {\em odd} if $\mathrm{arf}(q)=1$. The space of even (respectively, odd) quadratic forms is denoted by $Q(V)^+$ (resp., $Q(V)^-$). An easy computation gives $|Q(V)^+|=2^{g-1}(2^g+1)$ and $|Q(V)^-|=2^{g-1}(2^g-1)$.

Coming back to theta-characteristics, to any $\theta \in \mathrm{Th}(C)$ one associates the so-called {\em theta-form}
$q_{\theta}:J(C)[2]\rightarrow \mathbb{F}_2$,
by setting
$$q_{\theta}(\eta):=h^0(C,\eta \otimes \theta)+h^0(C,\theta) \,\mathrm{mod} \,2.$$
For any $\epsilon, \eta \in J(C)[2]$, the {\em Riemann-Mumford relation} (\cite{Vr1}) yields:
$$h^0(C,\theta \otimes \epsilon \otimes \eta)+h^0(C,\theta \otimes \epsilon)+h^0(C,\theta \otimes \eta)+h^0(C,\theta)\equiv\langle \eta, \epsilon \rangle \, \mathrm{mod}\, 2,$$
or equivalently, $q_{\theta}\in Q(J(C)[2])$. We thus identify the two spaces
\begin{equation}\label{id}\mathrm{Th}(C)=Q(J(C)[2]).\end{equation}
Since $\mathrm{arf}(q_{\theta})=h^0(C,\theta)\,\, \mathrm{ mod }\,\,2$ for any $\theta \in \mathrm{Th}(C)$, under this identification even (respectively, odd) theta-characteristics correspond to forms in $Q(J(C)[2])^+$ (resp., $Q(J(C)[2])^-$). Hence, on a genus $g$ curve $C$ there are precisely $2^{g-1}(2^g+1)$ even theta-characteristics and $2^{g-1}(2^g-1)$ odd ones.

Following Atiyah \cite{Atiyah}, a \emph{spin curve} of genus $g$ is a pair $(C,\theta)$, where $C$ is a smooth irreducible curve of genus $g$ and $\theta \in \mathrm{Th}(C)$. Mumford (\cite{Vr1}) and Atiyah (\cite{Atiyah}) proved that the parity of a spin curve, that is, the value of its Arf invariant, is locally constant in families. As a consequence, the moduli space $\mathcal{S}_g$ parametrizing spin curves of genus $g$ splits into two connected components $\mathcal{S}_g^+$ and $\mathcal{S}_g^-$: a pair  $(C,\theta)$ lies in $\mathcal{S}_g^+$ if $\theta$ is an even theta-characteristic, and in $\mathcal{S}_g^-$ otherwise.

For each $g,r\geq 0$ one can define the locus
$$\mathcal{S}_g^r:=\{(C,\theta) \in \mathcal{S}_g \ | \ h^0(C,\theta)\geq r+1 \ \mathrm{and} \ h^0(C,\theta)\equiv r+1 \ \mathrm{mod} \ 2\}.$$
Harris (\cite{Vr2}) proved that the dimension of every component of $\mathcal{S}_g^r$ is bounded below by the number $3g-3-{r+1 \choose 2}$, to which we refer as the expected dimension of $\mathcal{S}_g^r$. It may well be the case that  $\mathcal{S}_g^r$ is nonempty even when its expected dimension is negative. For instance, $\mathcal{S}_g^{[\frac{g-1}{2}]}\neq\emptyset$ contains hyperelliptic curves and is thus nonempty although in this case  $3g-3-{r+1 \choose 2}$ is very negative. Existence of components of $\mathcal{S}_g^r$ having the expected dimension has been established by Farkas  \cite{theta} in the range $1\leq r\leq 11, r\neq 10$ for all $g\geq g(r)$ where $g(r)$ is an  explicit integer.

In the case where $r=1$,  the locus $\mathcal{S}_g^1$ parametrizes curves having a so-called {\em vanishing thetanull}, that is, an effective even theta characteristic, and is a divisor in $\mathcal{S}_g^+$. Indeed, every theta characteristic $\theta$ on a general curve $C$ of genus $g$ has at most one section. This directly follows from the base-point-free pencil trick, which implies that, as soon as $H^0(C,\theta)$ contains a pencil $V$, the kernel of the Petri map 
$$\mu_{0,V}:V\otimes H^0(C,\theta)\rightarrow H^0(C,\omega_C)$$ is nonzero; indeed, one has $\mathrm{Ker}\mu_{0,V}\simeq H^0(C, \mathcal{O}_C(B))\neq 0$, where $B$ is the base locus of $V$. The Gieseker-Petri Theorem thus excludes the existence of such a pencil if the curve $C$ is general.

\subsection{Invariant vanishing thetanulls on étale double covers}\label{inv}
Let $\pi:\widetilde C\to C$ be the irreducible étale double cover associated with a general Prym curve $(C,\eta)$ of genus $g$. As explained at the end of Section \ref{uno}, the Brill-Noether varieties $W^r_{2g-2}(\widetilde C)$ governing the singularities of the theta divisor of $J(\widetilde C)$ are nonempty as soon as $\rho(2g-1,r,2g-2)\geq -r$; in particular, for values of $g$ for which there exists an $r$ yielding $-r\leq \rho(2g-1,r,2g-2)< 0$, the curve $\widetilde C$ is Brill-Noether special.

We now fix $g,r$ such that $\rho(2g-1,r,2g-2)\geq 0$; under this condition, one may still hope that $W^r_{2g-2}(\widetilde C)$ is smooth and of the expected dimension. However, this expectation fails even for $r=1$, as observed by Welters himself \cite[Rmk. 1.12]{Welters}; in particular, $\widetilde C$ never satisfies the Gieseker-Petri Theorem. By the base-point-free pencil trick, $W^1_{2g-2}(\widetilde C)\setminus W^2_{2g-2}(\widetilde C)$ is singular at any point defining a vanishing thetanull. On the other hand, a vanishing thetanull on $\widetilde C$ is easily obtained as pullback $\pi^*M$ of a theta-characteristics $M$ on $C$ such that both $M$ and $M\otimes \eta$ are odd. The following nice argument taken from \cite[Prop. 4]{Beauville} counts the number of such $M$s and thus in particular proves their existence. 

For any $M\in \mathrm{Th}(C)$, the push-pull formula yields
$$
h^0(\widetilde C,\pi^*M)=h^0(C,\pi_*\pi^*M)=h^0(C,M)+h^0(C,M\otimes \eta)=q_M(\eta)\,\mathrm{mod}\,2.
$$
We are looking for those $M\in \mathrm{Th}(C)$ such that $q_M(\eta)=0$ and $\mathrm{arf}(q_M)=1$. Pick $\epsilon \in  J(C)[2]$ such that the Weil pairing $\langle\eta,\epsilon\rangle=1$ and denote by $\Sigma\subset J(C)[2]$ the plane spanned by $\eta$ and $\epsilon$. Since $J(C)[2]=\Sigma\oplus \Sigma^\bot$, any $M\in \mathrm{Th}(C)$ is completely determined by the restrictions $q_M|_\Sigma$ and $q_M|_{\Sigma^\bot}$. The condition $q_M(\eta)=0$ yields $\mathrm{arf}(q_M|_\Sigma)=q_M(\eta)q_M(\epsilon)=0$. Therefore, $M$ is determined by the value $q_M(\epsilon)\in \mathbb F_2$ and by the restriction $q_M|_{\Sigma^\bot}$, which is a quadratic form on $\Sigma^\bot$ of arf invariant $1$ (because $\mathrm{arf}(q_M)=\mathrm{arf}(q_M|_\Sigma)+\mathrm{arf}(q_M|_{\Sigma^\bot})$). Since $\dim_{\mathbb F_2} \Sigma^\bot=2(g-1)$, we have $2^{g-2}(2^{g-1}-1)$ choices for $q_M|_{\Sigma^\bot}$ and, having the possibility of choosing $q_M(\epsilon)$, we obtain  $2^{g-1}(2^{g-1}-1)$ theta-characteristics $M$ on $C$ as above. Since $ \pi^*(M\otimes \eta)\simeq \pi^*M$, this construction provides precisely $2^{g-2}(2^{g-1}-1)$ vanishing thetanulls on $\widetilde C$, which are invariant under the covering involution. 

\section{Brill-Noether theory for double covers}\label{tre}
Given the étale double cover $\pi:\widetilde C\to C$ of a general Prym curve $(C,\eta)\in\mathcal R_g$, it is natural to investigate not only the Prym-Brill-Noether varieties $V^r(C,\eta)$ but any Brill-Noether variety $W^r_d(\widetilde C)$. However, not much is known about the Brill-Noether theory of $\widetilde C$. The following non-existence result follows from \cite{Schwarz}, where Schwarz proved a more general statement concerning étale cyclic covers of arbitrary degree.
\begin{thm}[\cite{Schwarz}]\label{irene}
Let $\pi:\tilde C\rightarrow C$ be the \'{e}tale double cover associated with a general Prym curve $(C,\eta)\in \mathcal R_g$. Then the Brill-Noether variety $W^r_d(\widetilde C)$ is empty if 
$$\rho(2g-1,r,d)<-r.$$ 
\end{thm}
The above result can be easily proved using Welters' degeneration and applying \cite[Prop. 4.1]{genus23} on the Brill-Noether theory of $2$-pointed elliptic curves.
For $r=1$ Theorem \ref{irene} is known to be optimal and this implies that $\widetilde C$ is not Brill-Noether general if $g$ is even. Indeed, the existence of a pencil of degree $g$ on $\widetilde C$ in this case follows from the surjectivity (cf. \cite{Arbarello}) of the difference map
\[
\begin{aligned}
\phi_{\frac{g}{2}} \colon C^{\frac{g}{2}} \times C^{\frac{g}{2}} &\longrightarrow J(C) \\
(D,E) &\longmapsto \mathcal{O}_C(D-E),
\end{aligned}
\]
 yielding $\eta=\mathcal O_C(D-E)$ for some effective divisors $D,E$ both of degree $g/2$ on $C$. The pullback $\pi^*(\mathcal O_C(E))\in \Pic^{g}(\widetilde C)$ then satisfies
 $$
 h^0(\widetilde C, \pi^*\mathcal O_C(E))=h^0(C,\pi_*\pi^*\mathcal O_C(E))= h^0(C, \mathcal O_C(E))+  h^0(C, \mathcal O_C(D))\geq 2,
 $$
 where we have again used the push-pull formula. On the other hand, in the odd genus case the gonality of $\widetilde C$ is maximal, as first proved by Aprodu and Farkas.
\begin{thm}[\cite{generalcovers} Thm. 0.4]
Let $\pi:\tilde C\rightarrow C$ be the \'{e}tale double cover associated with a general Prym curve $(C,\eta)\in \mathcal R_g$. Then the following hold:
\begin{itemize}
\item[(i)] if $g\equiv 1$ mod 2, the curve $\tilde C$ has maximal gonality, that is, $\mathrm{gon}(\tilde C)=g+1$;
\item[(ii)] if $g\equiv 0$ mod 2, then $\mathrm{gon}(\tilde C)=g$.
\end{itemize}
In both cases the Clifford index of $\widetilde C$ equals $\mathrm{gon}(\tilde C)-2$.
\end{thm}
Two different proofs are provided in \cite{generalcovers}, one by degeneration and one by specialization to curves on Nikulin surfaces; the latter are particular $K3$ surfaces on which we will focus in the next section. Up to our knowledge, the following natural question remains open:
\begin{question}
Let  $g$ be an odd positive integer such that the inequalities $-r\leq \rho(2g-1,r,2g-2)<0$ admit no integral solution $r\geq 1$. If $(C,\eta)\in \mathcal R_g$ is general, is the cover $\tilde C$ Brill-Noether general?
\end{question}

We stress that, quite surprisingly, ramified double covers of curves seem to behave better than étale ones from a Brill-Noether viewpoint. Following \cite{bud}, let $\mathcal R_{g,2n}$ be the moduli space parametrizing irreducible double covers of smooth genus $g$ curves ramified at $2n$ points. The space $\mathcal R_{g,2n}$ can be alternatively defined as follows:
$$
\mathcal R_{g,2n}:=\left\{(C,x_1,\ldots,x_{2n},\eta)\,|\, [C]\in\M_g,\,x_i\in C\,\forall\, i,\,\eta\in\Pic^{-n}(C),\, \eta^2=\mathcal O_C(-x_1-\ldots-x_{2n}) \right\}.
$$
While studying its birational geometry, Bud proved the following astonishing result, which is in contrast to the étale case.
\begin{thm}[\cite{bud} Thm. 4.1]
Let $\pi:\widetilde C\to C$ be the double cover associated with a general $(C,x+y,\eta)\in \mathcal{R}_{g,2}$. Then $\widetilde C$ is both Brill-Noether and Petri general.
\end{thm}
Bud's proof relies on the study of divisors on a compactification $\overline{\mathcal R_{g,2}}$ of the moduli space.  In the next section we will provide a simpler proof of the Brill-Noether generality of $\widetilde C$ by specialization to ramified double covers of curves on Nikulin surfaces of non-standard type, and we will extend the result to covers ramified at $4$ and $6$ points.

\section{Double covers of curves on Nikulin surfaces}\label{nik}
\subsection{Nikulin surfaces and their Picard group}Nikulin surfaces  represent a rather special class of K3 surfaces, which has
been studied in relation to various topics, including the theory of automorphisms \cite{VGS}, the study of Prym curves \cite{FK} and the birational
geometry of their moduli spaces \cite{FV1,FV,KLV1,KLV2}. We recall their definition and some basic properties.
\begin{df} \label{def:Nik}
  A polarized  {\it Nikulin} surface of genus $h \geq 2$ is a triple $(S,M,H)$ consisting of 
 a smooth $K3$ surface $S$ and two line bundles $\mathcal{O}_S(M),H \in \Pic S$ that satisfy the following conditions:
\begin{itemize}
\item $S$ contains $8$ disjoint smooth rational curves $N_1,\ldots,N_8$ such that $$ N_1+\cdots+N_8 \sim 2M. $$
\item $H$ is nef, $H^2=2(h-1)$ and  $H \cdot M=0$.
\end{itemize}
We say that $(S,M,H)$ is {\it primitively polarized} if the class of $H$ is primitive in $\Pic S$.
\end{df}
The line bundle $M$ defines a non-trivial double cover $\pi:\widehat{S}\rightarrow S$ branched along $\sum_{i=1}^8N_i$. Denoting by $\sigma:\widehat{S}\rightarrow \widetilde S$ the blow-down of the eight $(-1)$-curves $E_i:=\pi^{-1}(N_i)\subset \widehat{S}$, the surface $\widetilde S$ is a minimal $K3$ surface endowed with a so-called {\em Nikulin involution} $\iota \in \mathrm{Aut}(\widetilde S)$ having eight fixed points corresponding to the images $\sigma(E_i)$ of the exceptional divisors. The quotient $\bar S:=\widetilde S/\iota$ has eight ordinary double points and the quotient map $\bar \pi: \widetilde S\to \bar S$  fits in the following commutative diagram:

$$\xymatrix{ \widehat{S} \ar[d]_{\pi}\ar[r]^{\sigma}  & \widetilde{S} \ar[d]^{\bar{\pi}}\\
             S  \ar[r]_{\bar{\sigma}}& \bar{S}},$$
where $\bar \sigma$ is the contraction of the curves $N_i$ to the eight nodes of $\bar S$.

\begin{df} \label{def:Nik2}
Let $(S,M,H)$ be a polarized Nikulin surface of genus $h$. The rank $8$ sublattice of $\Pic S$ generated by $N_1,\ldots,N_8$ and $M$ is called {\it Nikulin lattice} and its denoted by $\mathbf{N}=\mathbf{N}(S,M)$.
\end{df} 
Since the Picard group of a polarized Nikulin surface $(S,M,H)$ contains both the Nikulin lattice and the polarization $H$, its rank is at least $9$. Consider the rank $9$ lattice
\[ \Lambda_h=\Lambda(S,M,H):= \ZZ[H] \oplus_{\perp} \mathbf{N} \subset \Pic S;\] 
the surface $S$ is said to be a {\it Nikulin surface of standard type} if the embedding $\Lambda_h \subset \Pic S$ is primitive, and  a {\it Nikulin surface of non-standard type} otherwise. Garbagnati and Sarti \cite{GS} proved that in the latter case $h$ is forced to be odd. The following result describes $\Pic(S)$ in the case of minimal Picard number.
\begin{prop}[\cite{GS}, Prop. 2.1]
Let $(S,M,H)$ be a genus $h$ primitively polarized Nikulin surface of Picard number $9$. Then either $\Pic S= \Lambda_h$ (standard case), or $h$ is odd and $\Lambda_h \subset \Pic S$ has index two (non-standard case). In the latter situation, possibly after renumbering the curves $N_i$, one falls in one of these cases:
\begin{itemize}
\item $h \equiv 3\,\mathrm{mod}\, 4$ and there are  $R_1, R_2 \in \Pic S$ such that $$R_1\sim \frac{H-N_1-N_2}{2},\,\,\,R_2\sim\frac{H-N_3-\cdots-N_8 }{2};$$ 
in particular, one has $g(R_1)=(h+1)/4$, $g(R_2)=(h-3)/4$.
\item  $h \equiv 1\,\mathrm{mod}\, 4$ and there are  $R_1, R_2 \in \Pic S$ such that $$R_1\sim \frac{H-N_1-N_2-N_3-N_4}{2},\,\,\,R_2\sim \frac{H-N_5-N_6-N_7-N_8}{2};$$
in particular, one has $g(R_1)=g(R_2)=(h-1)/4$.
\end{itemize}
\end{prop}
The growing interest in Nikulin surfaces is motivated from the fact that some Prym curves live on them. The condition $H\cdot M=0$ in Definition \ref{def:Nik} ensures that the double cover $\pi$ restricts to an étale double cover $\pi|_C:\widetilde{C}:=\pi^{-1}(C)\to C$ of any smooth curve $C\in |H|$; in other words, the pair $(C,M|_C)$ is a Prym curve of genus $h$. Since $\widetilde C$ is disjoint from the ($-1$)-curves $E_i$, the curve $\widetilde C$ can be identified with its image in $\widetilde S$ and we set $\widetilde H:=\mathcal O_{\widetilde S}(\widetilde C)$.

The Picard group of the $K3$ surface $\widetilde S$ was described by Van Geemen and Sarti  \cite{VGS} by investigating the action of the Nikulin involution $\iota$ on the cohomology group $H^2(\widetilde S,\mathbb Z)$. It turns out that $\Pic(\widetilde S)$ always contains the orthogonal complement $(H^2(\widetilde S,\mathbb Z)^\iota)^\perp$, which is isomorphic to the rank $8$ lattice $E_8(-2)$. Since $\widetilde H\in E_8(-2)^\perp\subset \Pic(\widetilde S)$, the Picard number of $\widetilde S$ is $\geq 9$ and equality holds if and only if the same holds for $S$. The following proposition summarizes results by Van Geemen and Sarti \cite[Prop. 2.2 and Prop. 2.7]{VGS} and Aprodu and Farkas \cite[Prop. 4.3]{generalcovers} in the standard case.
\begin{prop}\label{standard}
Let $(S,M,H)$ be a genus $h$ polarized Nikulin surface of standard type and Picard number $9$, and let $\widetilde S$ be the $K3$ surface with a Nikulin involution obtained from $S$. Then, the lattice
$$
\widetilde{\Lambda}_h:=\mathbb Z \widetilde H\oplus E_8(-2)$$
has index $2$ in $\Pic(\widetilde S)$. The latter is generated by $\widetilde{\Lambda}_h$ and a class $\frac{\widetilde H+v}{2}$, where $v$ is an element of $E_8(-2)$ satisfying $v^2=-4$ if $h$ is even and $v^2=-8$ if $h$ is odd. 
\end{prop}

We now turn to the non-standard case. By \cite[Prop. 3.5 (2)]{GS}, for $i=1,2$ the linear system $|R_i|$ contains a smooth irreducible curve as soon as $R_i$ has nonnegative self-intersection. We stress that $R_1\cdot M=2$ and $R_2\cdot M=6$ if $h \equiv 3\,\mathrm{mod}\, 4$, while $R_1\cdot M=R_2\cdot M=4$ for $h \equiv 1\,\mathrm{mod}\, 4$. Therefore, for all smooth curves $D_1\in |R_1|$ and $D_2\in |R_2|$, the cover $\pi$ induces double covers $\pi_{D_1}:\pi^{-1}(D_1)\to D_1$ and $\pi_{D_2}:\pi^{-1}(D_2)\to D_2$ that are ramified at $2$ and $6$ points, respectively, if $h \equiv 3\,\mathrm{mod}\, 4$; on the other hand, both $\pi_{D_1}$ and $\pi_{D_2}$ are ramified at $4$ points if $h \equiv 1\,\mathrm{mod}\, 4$. In any case, the curves $\pi^{-1}(D_1)$ and $\pi^{-1}(D_2)$ meet each exceptional curve $E_i$ in at most $1$ point and are thus isomorphic to their images in $\widetilde S$, that we denote by $\widetilde D_1 $ and $\widetilde D_2$ respectively. It is trivial to check that 
\begin{align}\label{pasqua}
\sigma^*\widetilde D_1-E_1-E_2\sim\pi^*D_1,\,\,\,\sigma^*\widetilde D_2-E_3-\cdots-E_8\sim\pi^*D_2\textrm{ if }h \equiv 3\,\mathrm{mod}\, 4,\\ 
\nonumber\sigma^*\widetilde D_1-E_1-E_2-E_3-E_4\sim\pi^*D_1,\,\,\,\sigma^*\widetilde D_2-E_5-E_6-E_7-E_8\sim\pi^*D_2\textrm{ if } h\equiv 1\,\mathrm{mod}\, 4. 
\end{align}
Note that Hurwitz's formula always yields $g(\widetilde D_1)=g(\widetilde D_2)=(h+1)/2$.\\
The following result directly follows from \cite[Prop. 2.2 and Prop. 2.7]{VGS} in the non-standard case.
\begin{prop}\label{non-standard}
Let $(S,M,H)$ be a genus $h$ polarized Nikulin surface of non-standard type and Picard number $9$, and let $\widetilde S$ be the $K3$ surface with a Nikulin involution obtained from $S$. Then
$$
\Pic(\widetilde S)=\mathbb Z \widetilde R\oplus E_8(-2),
$$
where $\widetilde R$ is a polarization of genus $(h+1)/2$ such that the curves $\widetilde D_1,\widetilde D_2$ as in \eqref{pasqua} lie in $|\widetilde R|$. 
\end{prop}

\subsection{Standard Nikulin surfaces}
As highlighted in \cite[Prop. 2.3]{KLV1}, a general curve $C\in |H|$ on a very general genus $h$ primitively polarized Nikulin surface $(S,M,H)$ of standard type is Brill-Noether general . Furthermore, it was proved by Aprodu and Farkas \cite[Thm. 1.5]{generalcovers} that the étale double cover $\widetilde C\subset \widetilde S$ of $C$ has the gonality of a general curve of genus $2h-1$ that covers a genus $h$ curve, namely, $h+1$ if $h$ is odd and $h$ otherwise. It is thus natural to ask whether $\widetilde C$ is general from a Prym-Brill-Noether viewpoint, that is, if it satisfies Welters' Theorem. A positive answer would provide an alternative proof of Welters' result avoiding degeneration, in analogy with Lazarsfeld's proof of the Gieseker-Petri Theorem by specialization to curves on $K3$ surfaces. Unfortunately, the answer turns out to be negative as soon $h>7$ and $h=6$. This bound on the genus agrees with the following theorem by Farkas and Verra.
\begin{thm}[\cite{FV}, Thm. 0.2]
A general Prym curve $(C, \eta) \in \mathcal{R}_h$ lies on a Nikulin surface if and only if $h\leq 7$ and $h\neq 6$.
\end{thm}
We thus prove the following result.
\begin{thm}
\label{Nikulin}
Let $(S,M,H)$ be a genus $h$ primitively polarized Nikulin surface of standard type with either $h>7$, or $h=6$. For any smooth curve $C\in |H|$, the double cover $\widetilde C\in |\widetilde H|$ of $C$ defined by $M|_C$ does not satisfy Welter's Theorem. 
\end{thm}
\proof
Set $A:=\frac{\widetilde{H}+v}{2}\in \Pic(\widetilde S)$ with $v\in E_8(-2)$ as in  Proposition \ref{standard}. Since $\widetilde{C}\cdot v=0$, the restriction $A|_{\widetilde C}$ is a line bundle on $\widetilde{C}$ of degree $2h-2$ such that:
$$
r:= h^0(\widetilde C,A|_{\widetilde C})-1= h^0(\widetilde S,A)-1\geq \chi(A)-1=1+\frac{1}{2}\left(\frac{\widetilde{H}+v}{2}  \right)^2=\left\lfloor \frac{h}{2}  \right\rfloor;
$$
here, the first equality follows from the strong version of Bertini's Theorem due to Saint-Donat \cite{SD} yielding $h^1(\widetilde S,A(-\widetilde C))=0$, while in the last equality  we have used that $v^2=-4$ when $h$ is even and $v^2=-8$ otherwise. Since $E_8(-2)=(H^2(\widetilde S,\mathbb Z)^\iota)^\perp\subset \Pic(\widetilde S)$, one has $\iota^*\widetilde H=\widetilde H$ and $\iota^*v=-v$. This implies $\iota^*A|_ {\widetilde C}\simeq\omega_{\widetilde C} \otimes  A|_{\widetilde C}^\vee$, or equivalently, $\mathrm{Nm}(A|_ {\widetilde C})=\omega_C$. Hence, $A|_ {\widetilde C}$ defines an element of the Prym-Brill-Noether variety $V^{r}(C,M|_C)$. 
If $h$ is odd, an easy computation gives
$$\rho^-(h,r)=-\frac{(h-1)(h-7)}{8},$$
which is negative when $h>7$. Analogously, for even $h$ one computes the Prym-Brill-Noether number
$$\rho^-(h,r)=-\frac{(h-2)(h-4)}{8},$$
which is negative for $h\geq 6$.  Hence, the Prym-Petri map of $A|_{\widetilde C}$ cannot be injective if either $h>7$ or $h=6$, and this concludes the proof.
\endproof

\subsection{Non-standard Nikulin surfaces}
Let $(S, M, H)$ be a very general genus $h$ primitively polarized Nikulin surface of non-standard type. As remarked in \cite[Prop.~2.3 and Rmk.~2.4]{KLV1}, in this case the line bundles $R_1,R_2\in \Pic(S)$ prevent curves in $|H|$ from being Brill-Noether general; indeed, their restrictions to any smooth curve $C\in |H|$ define two theta-characteristics on $C$ with negative Brill-Noether number. Proposition 3.5 (2) in \cite{GS} yields the existence of a smooth irreducible curve in the linear systems $|R_i|$ for $i=1,2$ as soon as $c_1(R_i)^2\geq 0$. Let us pick a smooth curve $D$ lying in either $|R_1|,|R_2|$, and denote by $g\geq 1$ its genus. If $h \equiv 1\,\mathrm{mod}\, 4$, then $g=(h-1)/4$ and $(D,M|_C)\in \mathcal R_{g,4}$. If instead $h \equiv 3\,\mathrm{mod}\, 4$, then either $g=(h+1)/4$ and $(D,M|_C)\in \mathcal R_{g,2}$, or $g=(h-3)/4$ and $(D,M|_C)\in \mathcal R_{g,6}$, depending on whether $D$ lies in $|R_1|$ or $|R_2|$. In any case the double cover $\widetilde D\subset \widetilde S$ of $D$ defined by $M|_C$ has genus $(h+1)/2$.
\begin{prop}\label{utile}
Let $(S, M, H)$ be a general non-standard Nikulin surface of odd genus $h$ and let $D$ be a smooth curve in either $|R_1|$ or $|R_2|$. Then the following hold:
\begin{itemize}
\item[(i)] $D$ is always Brill-Noether general, and it is Petri general if it is general in its linear system;  
\item[(ii)] the ramified double cover $\widetilde D$ of $D$ defined by $M|_C$ is Brill-Noether general.
\end{itemize}
\end{prop}
\begin{proof}
It is easy to verify that neither $R_1$ nor $R_2$ can be decomposed as the tensor product of two line bundles on $S$ both satisfying $h^0\geq 2$, and thus (i) follows proceeding as in Lazarsfeld's proof of Petri's Theorem \cite{Lazarsfeld, Pareschi}. Proposition \ref{non-standard} yields $\widetilde D\in |\widetilde R|$, where $\widetilde R$ is a generator of $\Pic(\widetilde S)$. Again one can easily exclude the existence of $L_1,L_2\in \Pic(\widetilde S)$ with $\widetilde R\simeq L_1\otimes L_2$ and $h^0(\widetilde S,L_i)\geq 2$. Hence, Lazarsfelds's Theorem implies that all smooth curves in the linear system $|\widetilde R|$ (thus, in particular, $\widetilde D$) are Brill-Noether general.
\end{proof}
A general curve $\widetilde D$ as in the above statement is highly expected to be Petri general, as well. However, this does not follow directly from Lazarsfeld's Theorem. Indeed, the latter only implies Petri generality for general curves in $|\widetilde R|$; however, a double cover $\widetilde D$ of a curve $D$ in $|R_1|$ or $|R_2|$ is never general in  the linear system $|\widetilde R|$, as it lies in either the $\iota$-invariant or the $\iota$-anti-invariant part of it.

By varying $h$, one obtains all possible values for the genus $g$ of $D$. Recalling that the moduli space $ \mathcal R_{g,2n}$ is irreducible for every $g\geq 2$ and $n\geq 0$  (cf. \cite[Prop. 2.5]{bud}), Proposition \ref{utile} then implies the following generalization of Bud's result.
\begin{thm}
Fix integers $g\geq 2$ and $n=1,2,3$. Let $(D,x_1,\ldots,x_{2n},\eta)\in \mathcal R_{g,2n}$ be general and let $\widetilde D$ the double cover of $D$ defined by $\eta$. Then the curve $\widetilde D$ is Brill-Noether general.
\end{thm}

\end{document}